\documentclass[11pt]{amsart}
\usepackage{amssymb}
\usepackage{amscd}
\usepackage{color}

\usepackage[all]{xy}

\newtheorem{theorem}{Theorem}[section]
\newtheorem{corollary}{Corollary}[section]
\newtheorem{lemma}{Lemma}[section]

\newtheorem{remark}{Remark}[section]

\newtheorem{proposition}{Proposition}[section]

\def \<{\langle}
\def \>{\rangle}

\newcommand{\bea}{\begin{eqnarray}}
\newcommand{\eea}{\end{eqnarray}}
\newcommand{\be}{\begin {equation}}
\newcommand{\ee}{\end{equation}}

\newcommand{\WW}{\mathcal{W}}
\newcommand{\W}{\mathcal{W}_3\left(\frac{4}{5}\right)}

\newcommand{\ds}{\displaystyle}

\newcommand{\nb}{{\mbox {\tiny{ ${\bullet \atop \bullet}$}}}}

\newcommand{\4}{\frac{4}{5}}
\newcommand{\8}{\frac{1}{8}}
\newcommand{\fth}{\frac{13}{8}}
\newcommand{\ZZ}{\mathbb{Z}}

\begin{document}
\title[The $3$-state Potts model and Rogers-Ramanujan series]  {The $3$-state Potts model and Rogers-Ramanujan series}
\author{Alex J. Feingold}
\author{Antun Milas}
\address{Department of Mathematical Sciences, SUNY-Binghamton}
\address{Department of Mathematics and Statistics, SUNY-Albany}
\email{alex@math.binghamton.edu}
\email{amilas@albany.edu}

\begin{abstract} We explain the  appearance of Rogers-Ramanujan series inside the tensor product of two basic  
$A_2^{(2)}$-modules, previously discovered by the first author in \cite{F}.
The key new ingredients are $(5,6)$ Virasoro minimal models and twisted modules for the Zamolodchikov $\WW_3$-algebra. 
\end{abstract}

\maketitle

\section{Introduction}
Virasoro minimal models and their characters have a long (and fruitful) history in conformal field theory, string theory and of course in vertex algebra theory.
It is by now well-known that the character $\chi_{s,t}^{m,n}(q)$ of the minimal model $L(c_{s,t},h_{s,t}^{m,n})$ is given by 
$$\chi_{s,t}^{m,n}(q) 
= \frac{q^{(h_{s,t}^{m,n}-c_{s,t}/24)}}{(q)_\infty} \cdot \displaystyle{\sum_{k \in \ZZ}} q^{st k^2}(q^{k(mt-ns)}-q^{(mt+ns)k+mn}),$$  
where $2\leq s,t\in\ZZ$ are coprime, $1\leq m < s$, $1\leq n < t$, $(q)_\infty = \phi(q) = \prod_{i= 1}^\infty (1 - q^i)$ and 
\be \label{cst}
c_{s,t}=1-\frac{6(s-t)^2}{st}, \ \ h_{s,t}^{m,n}=\frac{(mt-ns)^2-(s-t)^2}{4 st}.
\ee
In particular,
\bea \label{rr1}
&& \chi_{2,5}^{1,1}(q)=q^{11/60} \prod_{n =0}^\infty \frac{1}{(1-q^{5n+2})(1-q^{5n+3})}, \\
\label{rr2}
&& \chi_{2,5}^{1,2}(q)=q^{-1/60} \prod_{n = 0}^\infty \frac{1}{(1-q^{5n+1})(1-q^{5n+4})}.
\eea
give the product sides of the famous Rogers-Ramanujan series which appear 
in many papers on representation theory of Virasoro and affine Lie algebras.

The results in this note arose as an attempt to find a representation theoretic explanation of the following theorem obtained in 1983 
by the first author \cite{F}:
\begin{theorem} \label{alex} Denote by $L(\Lambda_1)$ the basic (level one) highest weight module for the twisted affine Kac-Moody algebra $A_2^{(2)}$. Then 
\be \label{basic-decomp}
L(\Lambda_1) \otimes L(\Lambda_1)=L(2 \Lambda_1) \otimes V_1 \oplus L(\Lambda_0) \otimes V_2,
\ee
such that  suitably normalized and scaled characters of the multiplicity spaces $V_1$ and $V_2$ coincide with the product sides of Rogers-Ramanujan series (\ref{rr1}-\ref{rr2}). 
\end{theorem}
This result was obtained in the principal ${\rm mod} \ 6$ realization of $A_2^{(2)}$ also studied in  \cite{Ca}, \cite{X}, \cite{Bo}, so it is perhaps not obvious what is its reformulation it in the language of vertex algebras and twisted modules. 
The above result is somewhat unexpected because the coset spaces, which are clearly unitary modules for the Virasoro algebra, are essentially given by the characters of the $(s,t) = (2,5)$ {\em non-unitary} minimal models! It is also interesting to notice that the principally specialized characters of  $L(2 \Lambda_1)$ and $L(\Lambda_0)$ are 
also given by above Rogers-Ramanujan series (again, suitably normalized and scaled).

To explain the appearance of Rogers-Ramanujan $q$-series we recall first a pair of identities which can be traced 
back to Bytsko and Fring \cite{BF}. (For further identities, see \cite{BF} and also \cite{Muk}, \cite{Mi1}):
\bea \label{minimal1}
&& \chi_{5,6}^{1,2}(q)+\chi_{5,6}^{1,4}(q)=\chi_{2,5}^{1,1}(q^{1/2}), \\
\label{minimal2}
&& \chi_{5,6}^{2,2}(q)+\chi_{5,6}^{2,4}(q)=\chi_{2,5}^{1,2}(q^{1/2}).
\eea
We also mention the related identities 
\bea
\label{minimal3}
&& \chi_{5,6}^{2,1}(q)-\chi_{5,6}^{2,5}(q)=\chi_{2,5}^{1,1}(q^{2}),  \\
\label{minimal4}
&& \chi_{5,6}^{1,1}(q)-\chi_{5,6}^{1,5}(q)=\chi_{2,5}^{1,2}(q^{2}).
\eea
Although all these formulas can be checked directly (see the Appendix) 
they are very interesting for several reasons. Their right hand sides are clearly (scaled) Rogers-Ramanujan series. 
Also, unlike the identities studied in \cite{Muk} and \cite{Mi1}, the above formulas are among characters of modules of {\em different} central charges (in our case, central charge $\frac{-22}{5}$ and  $\frac{4}{5}$). The key observation is now that the central charge of the cosets  spaces $V_1$ and $V_2$ equals $\frac{4}{5}$, the central charge of $(5,6)$ minimal models, 
So we are immediately led to the following conclusion:  relations (\ref{minimal1})-(\ref{minimal2}) should be interpreted as decomposition formulas of the coset spaces $V_1$ and $V_2$ into 
irreducible Virasoro characters of central charge $\frac{4}{5}$.  As we shall see this is indeed the case. In fact,  
there is something even deeper going on. It turns out that $V_1$ and $V_2$ are in fact
(the only twisted) irreducible modules for a larger rational vertex algebra which we end up calling $\W$, also known as the Zamolodchikov $\WW_3$-algebra. This algebra has already appeared in the 
physics literature under the name $3$-State Potts model and more recently in the vertex algebra theory \cite{KMY}. 
We first show that there are (at least)  four different ways of thinking about $\W$:

\begin{theorem} \label{W} The following vertex operator algebras are isomorphic:
\begin{itemize}
\item[(a)] The parafermionic space $K_{sl_2}(3,0) \subset L_{sl_2}(3,0)$ \cite{DW}.

\item[(b)] A certain subalgebra $M^0$ of  the lattice vertex algebra $V_L$, where $L=\sqrt{2}Q$ and  $Q$ is the root lattice of type $A_2$
(see \cite{KMY}, \cite{Miy}).

\item[(c)] The coset  vertex algebra $\WW^{L_{sl_3}(2 \Lambda_0)}_{L_{sl_3}(\Lambda_0) \otimes L_{sl_3}(\Lambda_0)}$ \cite{KW}.

\item[(d)] The simple affine $\WW$-algebra ${\bf L}(\4)$, obtained via Drinfeld-Sokolov reduction \cite{Ara}, \cite{FKW}.
\end{itemize}
\end{theorem}

In the physics literature equivalence of  $(a)$, $(c)$ and $(d)$ is more or less known. Construction $(b)$ is more recent (again, see \cite{KMY}).

Having enough knowledge about $\W$ we can now return to  the cosets $V_1$ and $V_2$. These are not ordinary modules for $\W$, but rather $\tau$-twisted $\W$-modules.
Our main result is the following theorem about $\W$.

\begin{theorem} \label{main-thm} For the rational vertex algebra $\W$ the following holds:
\begin{itemize}
\item[(a)] $Aut(\W)=\mathbb{Z}_2$.
\item[(b)] If we denote by $\tau$ the nontrivial automorphism of $\W$, then $\W$ is  $\tau$-rational (i.e. $\W$ has finitely many irreps and every $\tau$-twisted 
module is completely reducible).
\item[(c)] The algebra $\W$ has precisely two inequivalent representations, 
$$\WW^\tau(1/40)=L\left(\4,\frac{1}{40}\right) \oplus L\left(\4,\frac{21}{40}\right), \ \hbox{ and } $$
$$\WW^\tau(1/8)= L\left(\4,\frac{1}{8}\right) \oplus L\left(\4, \frac{13}{8}\right).$$
\item[(d)] $\chi(\WW^\tau(1/40)) = \chi_{2,5}^{1,2}(q^{1/2}) $ and $\chi(\WW^\tau(1/8)) = \chi_{2,5}^{1,1}(q^{1/2})$. 
\end{itemize}
\end{theorem}

Now we connect results from Theorem \ref{alex} with  those in Theorem \ref{main-thm}.
Denote by $\sigma$ the principal automorphism of order $6$ of $sl_3$, and let $L_{sl_3}(\Lambda_0)$ be the affine vertex algebra associated to  $\hat{sl_3}$ of level one.
The basic $A_2^{(2)}$-module $L^\sigma_{sl_3}(\Lambda_1)$ can be viewed as a $\sigma$-twisted $L_{sl_3}(\Lambda_0)$-module, and its character is given by
$${\rm tr}_{L(\Lambda_1)} q^{L^\sigma(0)-c/24}=q^{-1/72}\prod_{n = 0}^\infty \frac{1}{(1-q^{(6n+1)/6})(1-q^{(6n+5)/6})},$$
where $L^\sigma(0)$ is the $\sigma$-twisted Virasoro operator.

Similarly, the characters of the two standard level two $A_2^{(2)}$-modules (if viewed as $\sigma$-twisted $L_{sl_3}(2 \Lambda_0)$-modules)  are given by 
$${\rm tr}_{L(2 \Lambda_1)} q^{L^\sigma(0)-c/24}= q^{-1/72}\prod_{n = 0}^\infty \frac{1}{(1-q^{(6n+1)/6})(1-q^{(6n+5)/6})} \chi_{2,5}^{1,2}(q^{1/3})$$
and 
$${\rm tr}_{L(\Lambda_0)} q^{L^\sigma(0)-c/24}= q^{1/6-1/72}\prod_{n = 0}^\infty  \frac{1}{(1-q^{(6n+1)/6})(1-q^{(6n+5)/6})} \chi_{2,5}^{1,1}(q^{1/3}). $$
Now, formula (\ref{basic-decomp}) implies a $q$-series identity
\bea
 &\ds q^{-1/72} \prod_{n = 0}^\infty \frac{1}{(1-q^{(6n+1)/6})(1-q^{(6n+5)/6})} \nonumber \\
 &=\chi_{2,5}^{1,2}(q^{1/3})\chi_{2,5}^{1,2}(q^{1/2})+\chi_{2,5}^{1,1}(q^{1/3}) \chi_{2,5}^{1,1}(q^{1/2}), \nonumber 
\eea
after we canceled  ${\rm tr}_{L(\Lambda_1)} q^{L^\sigma(0)-c/24}$ from both sides of (\ref{basic-decomp}). This formula can be traced back to Ramanujan.
This gives the main result of our paper.
\begin{theorem}
As $ \sigma \times \tau$-twisted $L_{sl_3}(2 \Lambda_0) \otimes \W$-module:
\be 
L_{sl_3}^{\sigma}(\Lambda_1) \otimes L^\sigma_{sl_3}(\Lambda_1) \cong L_{sl_3}^\sigma(2 \Lambda_1) \otimes \WW^\tau(1/40)  \oplus L^\sigma_{sl_3}(\Lambda_0) \otimes \WW^\tau(1/8).
\ee
We also have 
\bea \label{char-coset}
&& {\rm tr}_{V_1} q^{L(0)-c/24} =\chi_{2,5}^{1,2}(q^{1/2}) \nonumber \\
&& {\rm tr}_{V_2} q^{L(0)-c/24} = \chi_{2,5}^{1,1}(q^{1/2}).
\eea
where $L(0)$ is the coset Virasoro operator for the tensor product.
\end{theorem}

Finally, we mention that  (\ref{minimal3})-(\ref{minimal4}) can be also explained in terms of representations of $\W$. This requires a modular invariance theorem for $\tau$-twisted modules \cite{DLM2}. This was pursued in the last section.

\begin{remark} The algebra $\W$ and some of its modules also appear in \cite{Ma}, where the coset Virasoro construction is applied to the investigation of
the branching rule decomposition of level-$1$ irreducible $E_6^{(1)}$-modules with respect to the affine subalgebra $F_4^{(1)}$. 
\end{remark} 

\section{The $3$-state Potts model vertex algebra}

Let us recall a few basic facts about the Virasoro algebra and its representation theory.
We use $M(c,h)$ to denote the Virasoro Verma module of central charge $c$ and lowest 
conformal weight $h$, and denote its lowest weight vector by $v_{c,h}$. We let $V(c,0) = M(c,0)/ \<L(-1)v_{c,0}\>$,
the vacuum vertex algebra. We denote by $L(c,h)$ the unique irreducible quotient of $M(c,h)$.

Recall (\ref{cst}). We will focus on the central charge $c_{5,6}=\frac{4}{5}$. It is well-known that $L(\frac{4}{5},0)$, viewed as a vertex algebra,
has (up to equivalence) precisely $10$ irreducible modules \cite{W}:
\bea \label{list-irr}
&& L\left(\4,0\right), L\left(\4,\frac{1}{8}\right), L\left(\4,\frac{2}{3}\right), L\left(\4, \frac{13}{8}\right), L\left(\4,3\right), \nonumber \\
&& L\left(\4, \frac{2}{5}\right), L\left(\4,\frac{1}{40}\right), L\left(\4,\frac{1}{15}\right),L\left(\4,\frac{21}{40}\right),L\left(\4,\frac{7}{5}\right). 
\eea

It is also known that  $L(\frac{4}{5},0) \oplus L(\frac{4}{5},3)$ can be equipped with a simple vertex operator algebra structure (cf. \cite{KMY}). 
Of course, one can always define a vertex operator algebra structure on $V \oplus M$, where $M$ is any module with 
integral grading  by defining the action of $M$ on $M$ to be trivial. But such vertex algebra is not simple. 
As we shall see, there is a unique vertex operator algebra structure on $L(\frac{4}{5},0) \oplus L(\frac{4}{5},3)$.

Let us also recall that the space of irreducible characters of $L(\frac{4}{5},0)$ is $10$-dimensional (i.e. the 
characters are linearly independent). 

\begin{proposition} \label{unique} Let $(V,Y,{\bf 1})$ be a vertex algebra  such that $V \cong L(\4,0) \oplus L(\4,3)$ as a module for the Virasoro 
algebra and such that $Y|_{L(\4,3) \otimes L(\4,3)} \neq 0$. If another vertex algebra $W \cong L(\4,0) \oplus L(\4,3)$ satisfies the same property, then 
$W \cong V$.
\end{proposition}
\begin{proof} We certainly have an isomorphism  $f=f|_{L(\4,0)} \oplus f_2|_{L(\4,3)}$ between $V$ and $W$ viewed as Virasoro modules ($f$ is unique up to a choice of two nonzero scalars). Because of $Y({\bf 1},x)={\rm id}$,  the map 
$f_1$ is uniquely determined sending the vacuum of  $V$ to the vacuum of $W$.  Observe that 
in $V$,  $\tilde{Y}_V:=Y|_{L(\4,3) \otimes L(\4,3)}$ defines an intertwining operator of type 
${ L(\4,0) \choose L(\4,3) \ \ L(\4,3) }$ (otherwise this  would contradict  $ \ I { L(\4,3) \choose L(\4,3) \ \ L(\4,3) }=0$). 
The same holds for $\tilde{Y}_W$.
According to \cite{FHL}, this intertwining operator is unique up to a nonzero constant. 
Thus, after identification via $f$, we can find  $\nu \neq 0$ such that $ \nu \tilde{Y}_V=\tilde{Y}_W$.  
Therefore $(u,v) \mapsto (f_1(u),\frac{1}{\lambda}f_2(v))$, $\lambda^2=\nu$ defines the wanted automorphism between $V$ and $W$.
\end{proof}

Existence of the vertex operator algebra satisfying conditions in Proposition \ref{unique} has been established in  \cite{KMY}. 
We shall denote it by $\W$. More precisely, we have \cite{KMY}

\begin{theorem} \label{miyamoto}The vertex algebra $\W$ is rational with the following irreducible modules (we also write their decompositions viewed as $Vir$-modules):
\bea 
&& \WW(0)=L\left(\4,0\right) \oplus L\left(\4,3\right),  \ \ \ \ \WW(2/5)=L\left(\4, \frac{2}{5}\right) \oplus L\left(\4,\frac{7}{5}\right), \nonumber \\
&& \WW(2/5,+)=L\left(\4,\frac{2}{3}\right), \ \ \ \ \ \ \ \ \ \ \ \  \WW(2/5,-)=L\left(\4,\frac{2}{3}\right), \nonumber \\
&& \WW(1/15,+)=L\left(\4,\frac{1}{15}\right), \ \  \ \ \  \ \ \ \ \WW(1/15,-)=L\left(\4,\frac{1}{15}\right). \nonumber
\eea
\end{theorem}

\noindent Now we prove the equivalence of (b) and (d) in Theorem \ref{W}.

\begin{theorem} Vertex algebras $\W$ and ${\bf L}(\4)$ are isomorphic.
\end{theorem}
\begin{proof}
Recall first that  ${\bf L}(\4)$ is the unique irreducible quotient of the universal affine $\WW$-algebra $M_{sl_3}(\4)$, modulo 
the maximal ideal. We can think of $M_{sl_3}(\4)$ as the vertex algebra obtained via Drinfeld-Sokolov reduction from 
a universal affine vertex algebra associated to $\widehat{sl_3}$, or constructed by using free fields via Miura transformation.
Either way, it is known that $M_{sl_3}(\4)$ is freely generated by the conformal vector 
$\omega$ and another {\em primary vector} of degree three, $w_{-1}{\bf 1}$, that is $L(0)w_{-1}{\bf 1}=3 w_{-1}{\bf 1}$ and $L(n) w_{-1}{\bf 1}=0$ for $n \geq 1$. 
It is also known that $Y(w_{-1}{\bf 1},x)w_{-1}{\bf 1}$ is nonzero both in $M_{sl_3}(\4)$ and in the 
quotient ${\bf L}_{vac}(\4)$.  Let us record 
$${\rm tr}_{M_{sl_3}(\4)} q^{L(0)}=\frac{1}{(q^2;q)_\infty (q^3;q)_\infty},$$
where $(a;q)_\infty=\prod_{i=0}^\infty (1-aq^i)$.
More importantly, by using \color{black} Frenkel \color{black}-Kac-Wakimoto's character 
formula \color{black} \cite{FKW} \color{black} (proven by Arakawa \cite{Ara}), we compute the character of  ${\bf L}_{}(\4)$ as
$${\rm tr}_{{\bf L}(\4)} q^{L(0)-1/12}= q^{-1/12}(q)^{-2}_{\infty} \sum_{m,n \in \mathbb{Z}} \sum_{w \in S_3} \epsilon_w q^{\frac{|5 w(\rho)+20n \alpha_1+20m \alpha_2-4 \rho|^2}{40}},$$
where $ S_3$ is the Weyl group of $sl_3$, $\alpha_1$ and $\alpha_2$ are simple roots, $\rho$ is the half-sum of positive roots and $|\lambda|^2=\langle \lambda,\lambda \rangle$ 
is normalized such that $|\alpha|^2=2$ for each simple root $\alpha$.
An easy computation shows that 
$${\rm tr}_{{\bf L}(\4)} q^{L(0)-1/12}=q^{-1/30}(1+q^2+2q^3+3q^4+ \cdots).$$
In fact we can show (by expanding both sides in $q$-series) that the following identity holds for $m \leq 50$:
\be \label{q-50}
{\rm tr}_{{\bf L}(\4)} q^{L(0)-1/24} \equiv \chi_{5,6}^{1,1}(q)+\chi_{5,6}^{1,5}(q) \ \ {\rm mod} \ q^{m},
\ee
Thus, the character of ${\bf L}(\4)$ equals the character of  $\W$ up to degree $50$. Let us consider  the Virasoro submodule $M \subset {\bf L}(\4)$  generated by 
$\omega$ and $w_{-1}{\bf 1}$. By Proposition \ref{unique}, it is sufficient to show
$U(Vir) \cdot {\bf 1} =L(\4,0) \subset {\bf L}(\4) $ and $U(Vir) \cdot w_{-1}{\bf 1}=L(\4,3)$. Indeed, if that is the case, then $L(\4,0)$ is a rational subalgebra of ${\bf L}(\4)$, which decomposes  as a sum of $L(\4,0)$-modules with integral conformal weights. But classification of $L(\4,0)$-modules and (\ref{q-50}) implies  ${\bf L}(\4) \cong L(\4,0) \oplus L(\4,3)$, and the rest of the proof now follows from Proposition \ref{unique}. 

Denote by $M_1$  the cyclic Virasoro module $U(Vir) \cdot {\bf 1}$ and by $M_2$  
the cyclic module $U(Vir) \cdot w_{-1}{\bf 1}$. By the universal property of  Verma modules, these cyclic modules are quotients of $M(\4,0)$ (even $V(\4,0)$) and $M(\4,3)$, respectively.  We claim that $M_1 \cap M_2=\{0 \}$. If not, $M_1 \cap M_2$ is a nontrivial submodule of a quotient of $M(\4,0)$ and of $M(\4,3)$. 
Embedding structure for Verma modules among the minimal series shows that this is impossible (the two modules belong to different blocks). 
If we let $\chi_W (q)={\rm tr}_{W} q^{L(0)-c/24}$, we get  
$$\chi_{{\bf L}(\4)}(q) \geq \chi_{M_1}(q) +  \chi_{M_2}(q) \geq   \chi_{5,6}^{1,1}(q)+ \chi_{5,6}^{1,5}(q),$$
where $\geq $ has an obvious meaning for two $q$-series with non-negative integer coefficients.
Now, relation (\ref{q-50}) implies $$\chi_{M_1}(q) +  \chi_{M_2}(q) \equiv   \chi_{5,6}^{1,1}(q)+ \chi_{5,6}^{1,5}(q) \ {\rm mod} \ q^{50}.$$
Thus $\chi_{M_1}(q) =\chi_{5,6}^{1,1}(q)$ and $\chi_{M_2}(q) = \chi_{5,6}^{1,4}(q)$, and hence  $U(Vir){\bf 1} =L(\4,0)$ and $U(Vir)w_{-1}{\bf 1}=L(\4,3)$.
\end{proof}

\begin{remark} The previous theorem gives a representation theoretic proof of the $q$-series identity
$$ q^{-1/12}(q)^{-2}_\infty \cdot  {\sum_{m,n \in \mathbb{Z}} \sum_{w \in S_3} \epsilon_w q^{\frac{|5 w(\rho)+20n \alpha_1+20m \alpha_2-4 \rho|^2}{40}}}
= \chi_{5,6}^{1,1}(q)+\chi_{5,6}^{1,5}(q),$$
which presumably can be checked directly by applying methods similar to those used in the Appendix. 
\end{remark}

\begin{theorem} \label{coset} Denote by $\WW^{sl_3(1) \otimes sl_3(1)}_{sl_3(2)}$ the coset vertex algebra obtained via 
the embedding $L_{sl_3}(2,0) \hookrightarrow L_{sl_3}(1,0) \otimes L_{sl_3}(1,0)$. 
Then, as vertex operator algebras, $\W \cong \WW_{sl_3(2)}^{sl_3(1) \otimes sl_3(1)}$
\end{theorem}
\begin{proof}
Observe first that the central charge of the coset vertex algebra is 
$$4-\frac{16}{5}=\frac{4}{5},$$
with the Virasoro generator 
$$\omega_1 \otimes {\bf 1} + {\bf 1} \otimes \omega_1-\omega_2,$$
where $\omega_k$ stands for the Sugawara generator of level $k \neq -3$.  
This coset is unitary so (as a Virasoro module,  or $L(\4,0)$-module) it decomposes as a direct sum of irreducible modules of central charge $4/5$. 
The  graded dimension can be now easily computed by using the rationality of $L(\4,0)$. Alternatively, we can recall a result from \cite{KW}, where the  
character of  $\WW^{sl_3(1) \otimes sl_3(1)}_{sl_3(2)}$  is computed by using modular invariance.
Either way, 
$$\WW^{sl_3(1) \otimes sl_3(1)}_{sl_3(2)}=L \left(\4,0 \right) \oplus L \left(\4,3 \right),$$
as Virasoro modules.
Finally to finish the proof we need 
$$Y_{L(\4,3) \otimes L(\4,3)} \neq 0$$
in the coset algebra. 
This was proven by Bowknegt et al. in \cite{BBSS} \cite{BS}, where it was shown that $\WW^{sl_3(1) \otimes sl_3(1)}_{sl_3(2)}$ has a degree $3$ generator 
$w_{-1}{\bf 1}$ such that the brackets $[w_n,w_m]$ satisfy the relations as in ${\bf L}(\frac{4}{5})$.
\end{proof}

\section{$\tau$-twisted $\W$-modules and the structure of $A_{\tau}(\W)$}

In this part $V$ is a vertex operator algebra and $\tau \in Aut(V)$ has order two. 
Recall the notion of a weak $\tau$-twisted $V$-module $M$.
By definition we require a decomposition $V=V^0 \oplus V^1$, and the twisted vertex operator map $Y^\tau(\cdot ,x)$
$$Y^\tau(u,x) \in ({\rm End} \ M)[[x^{1/2},x^{-1/2}]]$$
such that 
\be \label{tw-op}
Y^\tau(u,x)=\sum_{n \in \mathbb{Z}+\frac{r}{2}} u(n) z^{-n-1}, \ \ u \in V^r
\ee
$$u(n)v=0, \ \ n >> 0$$
$$Y^\tau({\bf 1},x)={\rm Id}_M$$
and the Jacobi identity holds.
\begin{align} 
x_0^{-1} \delta \left(\frac{x_1-x_2}{x_0} \right)Y^\tau(u,x_1) Y^\tau(v,x_2)-x_0^{-1} \delta \left(\frac{x_2-x_1}{-x_0} \right) Y^\tau(v,x_2) Y^\tau(u,x_1). \notag \\
\label{twj} =x_2^{-1} \left(\frac{x_1-x_0}{x_2}\right)^{r/2} \delta\left(\frac{x_1-x_0}{x_2} \right) Y^\tau(Y(u,x_0)v,x_2), 
\end{align}
where $u \in V^r$ and $v \in V^s$. 
Clearly, from the Jacobi identity it follows that $M$ is an untwisted (weak) $V^0$-module.
\color{black}

(1) If we also require $M$ to be graded $M=\oplus_{\nu} M_{\nu}$, where the grading is induced by the spectrum of $L(0)$ and is bounded from 
below, and has finite dimensional graded components, we say that $M$ is a $\tau$-twisted $V$-module. 

(2) If we instead require $M$ to be $\frac{1}{2} \mathbb{N}$-gradable, such that $M=\oplus_{n \in \frac{1}{2} \mathbb{N}} M(n)$, 
$u(m) M(n) \subset M(n +m - {\rm deg}(u)-1),$  then $M$ is called {\em admissible} \cite{DLM}. 
\color{black}

Let $M$ be an irreducible $\tau$-twisted $V$-module. Then there is $\lambda$ such that  we have the following decomposition with respect to the spectrum of $L(0)$:
\be \label{tw-gr}
M=\bigoplus_{n=0}^\infty (M_{\lambda + n} \oplus M_{\lambda+n+\frac{1}{2}}) ,
\ee
so that if we let $M^i=M_{\lambda+n+\frac{i}{2}}, \ \ i \in \{0,1 \}$, the vertex operator map
$Y^\tau( \cdot ,x)$ is compatible with this $\mathbb{Z}_2$-decomposition,  a consequence of 
$$u(m) M_{\nu} \subset M_{\nu +m - {\rm deg}(u)-1},$$
for homogeneous $u$.
Again, $M^i$ is a $V^0$-module  for $i \in \{0,1 \}$.
\color{black} A vertex operator algebra  $V$ is said to be $\tau$-rational if every {admissible} $\tau$-twisted $V$-module is completely reducible \cite{DLM}.
\color{black}

We also discuss intertwining operators among irreducible (ordinary) $V$-modules \cite{FHL}. Without giving the full definition, let us record that 
an intertwining operator of type 
$ { W_3 \choose W_1 \ \ W_2 }$ is a linear map $\mathcal{Y}(u,x) \in {\rm Hom}(W_2,W_3)\{x \}$, $u \in W_1$ such that 
(among other things) the following Jacobi identity holds:
$$x_0^{-1} \delta \left(\frac{x_1-x_2}{x_0} \right) Y(u,x_1) \mathcal{Y}(v,x_2)w-x_0^{-1} \delta \left(\frac{x_2-x_1}{-x_0} \right) \mathcal{Y}(v,x_2) Y(u,x_1)w$$
$$=x_2 ^{-1}\delta\left(\frac{x_1-x_0}{x_2} \right) \mathcal{Y}(Y(u,x_0)v,x_2)w,$$
where $u \in V$ , $v \in W_1$ and $w \in  W_2$. 

Let us now focus on the twisted Jacobi identity (\ref{twj})  when $u \in V^1$, and $v \in V^0$. We shall see that the corresponding identity is essentially intertwining 
operator map between $V^0$-modules.  Because $u \in V^1$, then $Y^\tau(u,x)$ restricted on $M^i$ is mapped to $M^{i+1}$ where we use the mod $2$ 
exponent notation. Denote by $I(u,x)$ this restriction. Since $v  \in V^0$ its action on $M^i$ is the usual $V^0$-module action,  so we write $Y_0(u,x)$ instead of $Y^\tau(u,x)$. 
The twisted Jacobi identity now reads (after we apply it on a vector $w$):
 
\begin{align}
x_0^{-1} \delta \left(\frac{x_1-x_2}{x_0} \right) I (u,x_1) Y_0(v,x_2)w-x_0^{-1} \delta \left(\frac{x_2-x_1}{-x_0} \right) Y_0(v,x_2) I(u,x_1) w \notag \\
\label{tw-int} =x_2^{-1} \left(\frac{x_1-x_0}{x_2}\right)^{1/2} \delta\left(\frac{x_1-x_0}{x_2} \right) I(Y(u,x_0)v,x_2)w, 
\end{align}

Now, apply the substitution $x_1 \leftrightarrow x_2$ and $x_0 \mapsto -x_0$. The identity is then 
\begin{align}
x_0^{-1} \delta \left(\frac{x_1-x_2}{x_0} \right) Y_0 (v,x_1) I (v,x_2)w-x_0^{-1} \delta \left(\frac{x_2-x_1}{-x_0} \right) I (u,x_2) Y_0(v,x_1)w \notag \\
\label{tw-int2} =x_1^{-1} \left(\frac{x_2+x_0}{x_1}\right)^{1/2} \delta\left(\frac{x_2+x_0}{x_1} \right) I(Y_0(u,-x_0)v,x_1)w.
\end{align}

Because of  $Y(u,x)v=e^{xL(-1)} Y_0(v,-x)u$ (the skew-symmetry), the right hand-side can be rewritten as

$$  x_1^{-1} \left(\frac{x_2+x_0}{x_1}\right)^{1/2} \delta\left(\frac{x_2+x_0}{x_1} \right) I(e^{-x _0 L(-1)} Y_0(v,x_0)u,x_1)w$$
\be \label{delta}
=  x_1^{-1} \left(\frac{x_2+x_0}{x_1}\right)^{1/2} \delta\left(\frac{x_2+x_0}{x_1} \right) I(Y_0(v,x_0)u,x_1-x_0)w.
\ee
To finish the proof observe first that,  by twisted weak associativity, we can always choose positive $k \in \mathbb{N}$ such that  
$$ x_2^{k+1/2} (x_2+x_0)^k I(Y_0(v,x_0)u,x_2)w$$
involves only positive (integral!) powers of the variable $x_2$. 
Consider 
\be \label{mod-jac-rhs}
  x_2^{-1}   \delta\left(\frac{x_1-x_0}{x_2} \right)  \left( x_1^{k}x_2^{k+1/2}  I(Y_0(v,x_0)u,x_2)w \right),
\ee
which also equals 
\be \label{mod-jac-rhs-1}
x_2^{-1}  \delta\left(\frac{x_1-x_0}{x_2} \right)  \left( (x_2+x_0)^{k}x_2^{k+1/2}   I(Y_0(v,x_0)u,x_2)w \right) ,
\ee
Now, we are allowed to replace in (\ref{mod-jac-rhs}) the $x_2$ variable with $x_1-x_0$, so we get 
\be \label{mod-jac-rhs-2}
  x_2^{-1}   \delta\left(\frac{x_1-x_0}{x_2} \right)  \left( x_1^{k}x_2^k (x_1-x_0)^{1/2}  I(Y_0(v,x_0)u,x_1-x_0)w \right).
\ee
Finally, we multiply the last expression with $x_2^{-k-1/2}x_1^{-k}$ and we obtain 
\bea 
&&  x_2^{-1}   \delta\left(\frac{x_1-x_0}{x_2} \right)   I(Y_0(v,x_0)u,x_2)w, \nonumber \\
&&=  x_2^{-1}   \delta\left(\frac{x_1-x_0}{x_2} \right)  \left(\frac{x_1-x_0}{x_2}\right)^{1/2}  I(Y_0(v,x_0)u,x_1-x_0)w. \nonumber
\eea

Consequently, (\ref{tw-int2}) and the last formula imply the following result.

\begin{proposition} \label{tw-to-int} The map $I ( \cdot, x)$ defines an intertwining operator among $V^0$-modules of type 
${ M^{i+1} \choose V^1 \ \ M^i }$. The identity (\ref{tw-int}) is equivalent to the Jacobi identity for $I( \cdot,x)$. 

\end{proposition}

We also have the following useful result
\begin{lemma} Let $M$ be an irreducible $\tau$-twisted $V$-module for a simple vertex algebra $V$. Then for every nonzero  $u \in V$, $m \in M$ 
$$Y^\tau(u,x) m \neq 0.$$
\end{lemma}
The proof is clear (otherwise $Y^{\tau}(a,x)m=0$ for every $a \in V$, which is impossible  because $M$ is cyclic). 
This statement, in particular, yields $M^1 \neq 0$ and $M^0 \neq 0$.

We shall need a few results about $\tau$-twisted Zhu's algebra $A_\tau(V)=V/O(V)$ following \cite{DLM},
where $O(V)$ is the span of vectors of the form 
$$a \circ b ={\rm Res}_x \frac{(1+x)^{{\rm deg}(a)-1+\delta_r+r/2}}{x^{1+\delta_r}} Y(a,x)b,$$
where $a \in V^r$ and where $\delta_0=1$ and $\delta_1=0$,
and the multiplication $*$ on $A_\tau(V)$ is induced via
$$a * b ={\rm Res}_x \frac{(1+x)^{{\rm deg}(a)}}{x} Y(a,x)b.$$
For every $a \in V^0$, and $b \in V^1$ we have $a*b=0 \ {\rm mod} \ O(V)$, so $b=0$ as an element in $A_\tau(V)$. 

Now we specialize $V=\W = L(\4,0) \oplus L(\4,3)$.

\begin{lemma} \label{help}
The $\tau$-twisted Zhu algebra $A_\tau(\W)$ is a quotient of the polynomial algebra $k[x]$.
\end{lemma}

\begin{proof} Denote by $\pi$ the natural projection from $\W$ to $A_\tau(\W)$. Then $\pi(v)=0$ for $v \in A(\W)_1$ \cite{DLM}.  If we denote 
by $A_0(\W)=A(\W)$, where $A(\W)$ is the usual Zhu's algebra of $\W$,
we clearly have an isomorphism $A_\tau(\W) \cong A(\W)_0/I$ where $I$ is a certain ideal.
\end{proof}

We are primarily interested in irreducible $\tau$-twisted $\W$-modules.
Every such module is an ordinary module for  $L(\4,0)$, so it decomposes as a direct sum of (ordinary) modules given on the list in (\ref{list-irr}).
Because of (\ref{tw-gr}), any irreducible $\tau$-twisted $\W$-module is heavily constrained with respect to the spectrum of $L(0)$. More precisely, 

\begin{proposition} \label{irr-mod-shape} Let $M$ be an irreducible $\tau$-twisted $\W$-module as in (\ref{tw-gr}). Then $M$, viewed as a $L(\4,0)$-module, 
is isomorphic to either 
$$\bigoplus_{i \in I} L\left(\4,\frac{1}{40}\right) \oplus \bigoplus_{j \in J} L\left(\4,\frac{21}{40}\right)$$
or
$$\bigoplus_{k \in K} L\left(\4,\frac{1}{8}\right) \oplus \bigoplus_{l \in L} L\left(\4,\frac{13}{8}\right),$$
where $I,J,K,L$ are finite sets. 
\end{proposition}

\begin{proof} The spectrum of $L(0)$ on $M=M^0 \oplus M^1$ must be contained inside the set $\lambda+\frac{\mathbb{N}}{2}$, where $\lambda$ is  the lowest 
weights of an irreducible $L(\frac{4}{5},0)$-modules (\ref{list-irr}). Also, in addition $M^i \neq 0$. This implies, that the absolute value of the difference between the lowest conformal weights of $M^0$ and of $M^1$  must lie within the set $\mathbb{N}+\frac{1}{2}$. 
Easy inspection of allowed weights gives two possibilities: $\lambda=\frac{3}{8}$ or $\lambda=\frac{1}{40}$.
\end{proof}

The next result follows from the fusion rules for $L(\4,0)$ \cite{KMY}, \cite{W} (see also \cite{IK}).
\begin{lemma} \label{simple-curr}
The module $L(\4,3)$ is a simple current (i.e. it permutes equivalence classes of irreducible modules under the fusion product $\times$). In particular, 
$$L \left(\4,3 \right) \times L \left(\4,\frac{13}{8}\right) =L \left(\4,\frac{1}{8}\right) ,L \left(\4,3 \right) \times L \left(\4,\frac{1}{8}\right) =L \left(\4,\frac{13}{8}\right) ,$$
$$L \left(\4,3 \right) \times L \left(\4,\frac{1}{40}\right) =L \left(\4,\frac{21}{40}\right) ,L\left(\4,3 \right) \times L \left(\4,\frac{21}{40}\right) =L \left(\4,\frac{1}{40}\right).$$
\end{lemma}

Now let us examine the number of irreducible summands in decompositions in Proposition \ref{irr-mod-shape}.

\begin{proposition} \label{irr-prec} Let $M$ be as in Proposition \ref{irr-mod-shape}. Then, viewed as a $Vir$-module, 
$$M \cong L\left(\4,\frac{1}{40}\right) \oplus L\left(\4,\frac{21}{40}\right)$$ or 
$$M \cong L\left(\4,\frac{1}{8}\right) \oplus L\left(\4,\frac{13}{8}\right).$$
\end{proposition}

\begin{proof}
We prove discuss the first assertion. Because the twisted module in question is irreducible and $A_\tau(\W)$ is commutative the top level must be an (irreducible) one-dimensional module. Therefore the set $I$ must be 
a singleton. If $|J| \geq 2$, the direct sum $\bigoplus_{j \in J} L(\4,\frac{21}{40})$  decomposes  into at least two irreducible $L(\4,0)$-modules. 
To rule out this case we apply the same argument as in the proof of  Lemma 5.3 \cite{KMY}. Their argument and Lemma \ref{simple-curr} 
yields a $\W$-module on $L\left(\4,\frac{21}{40}\right)$, such that $L(\4,0)$ acts trivially on it. But this is clearly a contradiction, so $|J|=1$.

\end{proof}

\begin{theorem} \label{unique-decomp}
Let $M$ be an irreducible $\tau$-twisted $\W$-module such that 
$$M \cong L\left(\4,\frac{1}{40}\right) \oplus L\left(\4,\frac{21}{40}\right)$$ or 
$$M \cong L\left(\4,\frac{1}{8}\right) \oplus L\left(\4,\frac{13}{8}\right).$$
Then such $M$ is unique up to isomorphism.
\end{theorem}

\begin{proof}
Recall that irreducible $\W$-modules are in one-to-one correspondence with the modules for the Zhu algebra $A(\W)$, which is isomorphic to 
a quotient of the polynomial algebra $k[x]$, where $x=[\omega]$. Every such module is one-dimensional so the top level of an irreducible 
$\W$-module must be one-dimensional.The rest follows from Proposition \ref{irr-prec}
\end{proof}

In the next section (cf. Proposition \ref{decomp-tw})  we construct two (irreducible) $\tau$-twisted $\W$-modules $\WW^\tau(\frac{1}{40})$ and $\WW^\tau(\frac{1}{8})$,
which decompose as modules in Theorem \ref{unique-decomp}. Consequently, combined with Lemma \ref{help} 
we immediately obtain

\begin{corollary} \label{classify} The vertex algebra $\W$ has precisely two $\tau$-twisted irreducible modules. 
Moreover, the twisted Zhu algebra $A_{\tau}(\W)$  is isomorphic to $\mathbb{C}[x]/ \langle (x-1/40)(x-1/8) \rangle$.
\end{corollary}

\begin{theorem} The vertex algebra $\W$ is $C_2$-cofinite and $\tau$-rational.
\end{theorem}

\begin{proof}
The vertex algebra $\W$ contains a $C_2$-cofinite subalgebra (i.e. $L(\4,0)$, with the same conformal vector) thus it is $C_2$-cofinite itself.

To prove the rationality we follow the standard arguments as in Theorem 5.6 in \cite{KMY}, which we essentially repeat here.
We only have to prove complete reducibility. So let $M$ be an arbitrary admissible $\W$-module.
We may split $M=M_{(1/8)} \oplus M_{(1/40)}$, where the weights of $M_{(1/8)}$ are contained inside $1/8+\frac{1}{2} \mathbb{N}_{\geq 0}$, 
and the weights of  $M_{(1/40)}$ are in $1/40+\frac{1}{2} \mathbb{N}_{\geq 0}$. Indeed, this follows from complete reducibility with respect to $L(\4,0)$, Lemma \ref{simple-curr}, Proposition \ref{tw-to-int} and the fact that allowed lowest weights are $\frac{1}{8}$ and $\frac{1}{40}$.
Our proof of complete reducibility of $M_{(1/8)}$ is essentially the same as the proof for $M_{(1/40)}$, so let us assume $M=M_{(1/8)}$ for simplicity.
Consider the (top) weight $1/8$ subspace $M(1/8)$ of $M_{(1/8)}$. This is also an $A(\W)$-module. Easy analysis shows that, as $L(\4,0)$-module, $M \cong L(\4,\8)^{\oplus^m} \oplus L(\4,\fth)^{\oplus^n}$, with some multiplicities $m$ and $n$. The multiplicity $m$ is precisely the dimension of the weight $1/8$ subspace. 
To finish the proof we have to argue that $m=n$ and $M \cong (L(\4,\8) \oplus L(\4,\fth))^{\oplus^m}$, as $\W$-modules.
Choose $0 \neq v \in M(1/8)$ and consider $\W \cdot v$. We claim that $\W \cdot v \cong L(\4,\8) \oplus L(\4,\fth)$. Clearly there could be only
one copy of $L(\4,\8)$ inside $\W \cdot  v$. Also, from the fusion rules Lemma \ref{simple-curr}, restriction of the module map $Y|_{L(\4,3) \otimes L(\4,\8)}$ where 
 $L(\4,\8) \subset \W \cdot v$, must land inside $W$, where $W$ is isomorphic to $L(\4,\fth)$, or is plainly zero. If the image is zero then $L(\4,\8)$ becomes an irreducible module for $\W$, contradicting our classification of $\W$-modules in Proposition \ref{irr-prec}.
We conclude $\W \cdot v  \cong L(\4,\8) \oplus L(\4,\fth)$. 

Now, take a nonzero vector $v' \in M(1/8)$ in the complement of $\mathbb{C} v$ and repeat the procedure.
Then $\W \cdot v' \cap \W \cdot v$ must be trivial or $L(\4,\fth)$. The latter case cannot occur because of Proposition \ref{irr-prec}.
This way we obtain a decomposition of submodule of $M$ isomorphic to $(L(\4,\8) \oplus L(\4,\fth))^{\oplus^m}$.
The condition $m=n$ must be satisfied, otherwise we could quotient $M$ with the submodule $(L(\4,\8) \oplus  L(\4,\fth))^{\oplus^m}$ and obtain 
a module of  lowest weight $\frac{13}{8}$, again a contradiction. \end{proof}

\section{Standard $A_2^{(2)}$-modules and construction of $\tau$-twisted $\W$-modules}

In this section, our focus is the Kac-Moody Lie algebra of type $A_2^{(2)}$ and its standard modules. We will denote by $\sigma$ a principal (order $6$)  automorphism of $\widehat{sl_3}$, such that 
$$sl_3=\bigoplus_{j=0}^5 sl_3[j], \ \ sl_3[j]=\{ a \in sl_3 : \sigma(a)=\xi^j a \},$$
where $\xi$ is a primitive $6$th root of unity.
Let $\Lambda$ be a dominant weight of level $l$ for $A_2^{(2)}$. We denote by $L^\sigma_{sl_3}(\Lambda)$ the $\sigma$-twisted 
$A_2^{(2)}$-module, which is also an $L_{sl_3}(l \Lambda_0)$-module \cite{Li}. 
In particular, $L^\sigma_{sl_3}(\Lambda_2)$ (which is of level $2$)
and $L^\sigma_{sl_3}(2 \Lambda_1)$ are $\sigma$-twisted $L_{sl_3}(2 \Lambda_0)$-modules. Theorem \ref{alex} gives construction 
of both modules inside the  tensor product of the basic module $L_{sl_3}^\sigma(\Lambda_1) \otimes L_{sl_3}^\sigma(\Lambda_1) $. Indeed,  
it is easy to see that $v_{\Lambda_1} \otimes v_{\Lambda_1}$ is a highest weight module for $L_{sl_3}^\sigma(2 \Lambda_1)$ and 
the vector $f_1 \cdot v_{\Lambda_1} \otimes v_{\Lambda_1}- v_{\Lambda_1} \otimes f_1 \cdot  v_{\Lambda_1}$ generates the module $L_{sl_3}^\sigma(\Lambda_0)$. Here $\{e_0,f_0,h_0,e_1,f_1,h_1 \}$ is the canonical
set of generators of $A_2^{(2)}$.   We denote the twisted module map with $Y^\sigma( \cdot, x)$.
In particular, for any  $x \in sl_3[j]$, we have 
$$Y^\sigma(x(-1){\bf 1},x)=\sum_{n \in \mathbb{Z}} x(n+\frac{j}{6})x^{-n-\frac{j}{6}-1}.$$

Consider
$$\omega_{k=l}=\frac{1}{2(l+3)} \sum_{i=1}^8 u_i(-1)\bar{u}_i(-1){\bf 1} \in L_{sl_3}(l \Lambda_0),$$
the Sugawara conformal vector of central charge $\frac{8l}{l+3}$, where $\{ u_i \}$ and $\{\bar{u}_i \}$ are conveniently chosen orthogonal bases
such that $\sigma(u_i)=\sigma^{-1}(\bar{u}_i)$, so that $u_i(-1)\bar{u}_i(-1){\bf 1}$ is fixed under the automorphism (we can actually choose $\{ \bar{u}_i \}$ 
to be a permutation of $\{ u_i \}$ \cite{Bo}).  Thus 
$$Y(\omega_{l},x)=\sum_{n \in \mathbb{Z}} L^\sigma_{k=l}(0) x^{-n-2}.$$
Denote by 
$$\omega=\omega_{k=1} \otimes {\bf 1} - {\bf 1} \otimes \omega_{k=1}-\omega_{k=2} \in L_{sl_3}(\Lambda_0) \otimes L_{sl_3}(\Lambda_0)$$
the coset Virasoro generator of central charge $2+2-\frac{16}{5}=\frac{4}{5}$, and we let
$$Y(\omega,x)=\sum_{n \in \mathbb{Z}} L(n) x^{-n-2}.$$

For related coset constructions see \cite{AP}. The following lemma is the crucial technical fact.

\begin{lemma}  \label{hw}
We have $$L^\sigma_{k=1}(0) v_{\Lambda_1}=\frac{5}{72} v_{\Lambda_1},$$
and 
$$L^\sigma_{k=2}(0) v_{2 \Lambda_1}=\frac{41}{360}v_{2 \Lambda_1}, \ \ L^\sigma_{k=2}(0)v_{\Lambda_0}=\frac{13}{180} v_{\Lambda_0}.$$
Consequently, 
$$L(0)v_{2 \Lambda_1}=\frac{1}{40} v_{2 \Lambda_1}, \ \ L(0) v_{\Lambda_0}=\frac{1}{8} v_{\Lambda_0}$$
\end{lemma}
\begin{proof}
We prove only the first formula $L(0)v_{2 \Lambda_1}=\frac{1}{40} v_{2 \Lambda_1}$, the other formula is proven along the same lines.
Recall that for the Virasoro algebra operator $L(0)$ we picked generators  $u_i$, $\bar{u}_i$, $i \in \{1,...,8 \}$ such that 
$$\sigma(u_i)=\xi^j u_i, \ \ \sigma(\bar{u}_i)=\xi^{6-j} \bar{u}_i$$
Thus we have to compute expressions 
$${\rm Coeff}_{x^{-2}} Y^\sigma(u_i(-1)\bar{u}_i(-1){\bf 1},x),$$
contributing to $L(0)$,  
acting on the highest weight vectors $v_{\Lambda_1}$ and on $v_{2 \Lambda_1}=v_{\Lambda_1} \otimes v_{\Lambda_1}$. For that we use 
a version of the Jacobi identity (\ref{twj}) with $u=u_i(-1){\bf 1}$ and $v=\bar{u}_i(-1){\bf 1}$, where the automorphism $\tau$ is now $\sigma$.
In this setup 
$$ {\rm Res}_{x_2} {\rm CT}_{x_0} {\rm Res}_{x_1} ({\rm RHS} \ of \  (\ref{twj}))$$
$$={\rm Res}_{x_2} {\rm CT}_{x_0} {\rm Res}_{x_1} x_2^{-1} \left(\frac{x_1-x_0}{x_2}\right)^{-j/6} \delta\left(\frac{x_1-x_0}{x_2} \right) Y^\sigma(Y(u_i,x_0)\bar{u}_i,x_2),$$ 
which equals (by \cite{Bo}, \cite{Li})
\be \label{rhstw}
{\rm Res}_{x_2} x_2 Y^\sigma(u_i(-1)\bar{u}_i(-1){\bf 1},x_2) +(j/6) [u_i,\bar{u_i}](0)+{ j/6 \choose 2} l,
\ee
where $l$ is the level.

Taking the same residues, now of the left hand side of the Jacobi identity,  gives
\be \label{lhstw}
{\rm Res}_{x_2} {\rm CT}_{x_0} {\rm Res}_{x_1} ({\rm LHS}  \ of \  (\ref{twj}) )={\rm CT}_{x_2} \nb Y^\sigma(u_i,x_2) Y^\sigma(\bar{u}_i,x_2) \nb,
\ee
the constant term of a twisted normally ordered product.
Comparing the formulas (\ref{rhstw}) and (\ref{lhstw}), and summing over $i$,  gives an expression for $L^\sigma_{k=l}(0)$.
It is now easy to get $L^\sigma_{k=1}(0) v_{\Lambda_1}=\frac{5}{72} v_{\Lambda_1}$. Next, we use 
the Sugawara operator $L^\sigma_{k=2}(0)$ and act on $v_{2 \Lambda_1}$. The only nonzero contributions 
when calculating these operators come from $(j/6) [u_i,\bar{u_i}](0)$ and ${ j/6 \choose 2} l$. After we sum over $i$ we get
$L^\sigma_{k=2}(0) v_{2 \Lambda_1}=\frac{41}{360}v_{2 \Lambda_1}$ and finally $L(0)v_{2 \Lambda_1}=\frac{1}{40} v_{2 \Lambda_1}$.
\end{proof}

\begin{remark} Presumably the computation in Lemma \ref{hw} can be carried out via an explicit realization of $A_2^{(2)}$ on the twisted Fock space 
obtained in \cite{F}.
\end{remark}

Next result comes immediately from Theorem \ref{coset} and \cite{KW}:

\begin{proposition} \label{untw-decomp} We have the following decomposition of $L_{sl_3}(2,0) \otimes \W$-modules:
$$L_{sl_3}(\Lambda_0) \otimes L_{sl_3}(\Lambda_0)=L_{sl_3}(2 \Lambda_0) \otimes \WW(0) \oplus L_{sl_3}(\Lambda_1) \otimes \WW(2/5). $$
\end{proposition}

The next goal is to find a twisted version of Proposition \ref{untw-decomp}.
The automorphism $\sigma$ acts diagonally on $L_{sl_3}(\Lambda_0) \otimes L_{sl_3}(\Lambda_0)$, and is denoted by $\sigma \times \sigma$. We have to see how it behaves when 
restricted to the subalgebra $L_{sl_3}(2 \Lambda_0) \otimes \W$.

\begin{lemma} \label{aut} The automorphism $\sigma$ preserves $\W$.  More precisely, we have $\sigma|_{\W}=\tau$. 
Thus, $\sigma \times \sigma|_{L_{sl_3}(2 \Lambda_0) \otimes \W}=\sigma \times \tau$.
\end{lemma}

\begin{proof} 
The automorphism $\sigma$ acts (diagonally) on the tensor product  \\ $L_{sl_3}( \Lambda_0) \otimes L_{sl_3}( \Lambda_0)$.
Therefore,  $\sigma$ also preserves $L_{sl_3}(2 \Lambda_0)$. Since $\W$ is the commutant of 
$L_{sl_3}(2 \Lambda_0)$, by definition $a_m b=0$ for all $a \in L_{sl_3}(2 \Lambda_0)$ and $b \in \W$, $m \geq 0$.  But then $\sigma(a_m b)=\sigma(a)_m \sigma(b)=0$, and hence $\sigma(b) \in \W$. Recall  $Aut(\W)=\mathbb{Z}_2$. If  $\sigma|_{\W}=1$, then $V_1$ would be an ordinary $\W$-module.  Easy inspection of modules
in Theorem \ref{miyamoto} implies that this is impossible. The proof follows.
\end{proof}

Consequently, we reached a desired decomposition analogous to the one in Proposition \ref{untw-decomp}.
\begin{proposition} \label{decomp-tw}
As $\sigma \times \tau$-twisted  $L_{sl_3}(2,0) \otimes \W$-modules, 
$$L^\sigma_{sl_3}(\Lambda_1) \otimes L^\sigma_{sl_3}(\Lambda_1)=L^\sigma_{sl_3}(2 \Lambda_1) \otimes \WW^\tau \left(\frac{1}{40} \right) \oplus L^\sigma_{sl_3}(\Lambda_0) \otimes \WW^\tau \left(\frac{1}{8} \right),$$
where  $\WW^\tau \left(\frac{1}{40} \right)$ and $\WW^\tau \left(\frac{1}{8} \right)$ have lowest conformal weights $\frac{1}{40}$ and $\frac{1}{8}$, respectively.
\end{proposition}

\begin{proof}
The vertex algebra $L_{sl_3}(2 \Lambda_0)$ is $\sigma$-rational, thus $L^\sigma_{sl_3}(\Lambda_1) \otimes L^\sigma_{sl_3}(\Lambda_1)$ decomposes
as a direct sum of  $\sigma$-twisted $L_{sl_3}(2 \Lambda_0)$-modules.  This decomposition is described in Theorem \ref{alex}. 
As  $\W$ is the commutant of  $L_{sl_3}(2,0) \subset L_{sl_3}( \Lambda_0) \otimes L_{sl_3}(\Lambda_0) $, and $L^\sigma_{sl_3}(\Lambda_1) \otimes L^\sigma_{sl_3}(\Lambda_1)$ is a twisted  $\sigma \times \tau$-modules by Lemma \ref{aut}, the multiplicities spaces in Theorem \ref{alex} are naturally  $\tau$-twisted $\W$-modules.
The proof now follows from Lemma \ref{hw}.
\end{proof}

\section{Modular invariance}

Now we are ready to "explain" relations (\ref{minimal3}) and (\ref{minimal4}), and in particular the negative sign
appearing in both identities. First we recall a result from \cite{DLM2}, where a version of Zhu's of modular invariance theorem
was extended to general  $C_2$-cofinite rational $\tau$-twisted vertex algebras (here $\tau$ is at first an automorphism of finite order). 

The setup is as following. Pick a pair of commuting 
automorphisms $(g,h)$ (of finite order) of $V$, where $V$ is $C_2$-cofinite, and satisfies all the rationality and finiteness conditions as in \cite{DLM2}. Then we let
$$T_M(g,h,v,q)={\rm tr}_M \phi(h) o(v) q^{L(0)-c/24},$$
where $M$ is $h$-stable $g$-twisted sector. This holomorphic function in $|q|<1$, $q=e^{2 \pi i y}$, 
satisfies the modular transformation property under $\gamma = \left[ \begin{array}{cc} a & b \\ c & d \end{array} \right]$, $ad-bc=1$:

$$T_M(g,h,v,q)|_{\gamma \cdot q }=\sum_{W} \sigma_W T_W((g,h)\gamma,q)$$ 
where the summation goes over all $g^a h^c$-twisted sectors $W$ which are $g^b h^d$-stable, and 
$(g,h)  \gamma=(g^a h^c,g^b h^d)$. Observe that the modular invariance mixes several twisted 
sectors..

Now specialize $g=h=\tau$, where $\tau$ is of order two. 

\noindent Claim: $\tau \circ M \cong M$ for every irreducible $\tau$-twisted module $M$.

\noindent The $\tau$-twisted module $\tau \circ M$ is defined via 
$$\tilde{Y}(u,x)=Y^\tau(\tau u,x).$$
Because $\tau \circ M$ is $M=M^0 \oplus M^1$ as a vector space, we let  
$$\sigma : M \rightarrow \tau \circ M, \ \ \sigma|_{M^0}=1, \ \ \sigma_{M^1}=-1$$
We claim that $\sigma$ is the wanted isomorphism. 
If $u \in V^0$,  $\sigma (Y^\tau(u,x)m))=Y^\tau(u,x) \tau m=Y^\tau(u,x) \sigma m$, and if $u \in V^1$, 
$\sigma (Y^\tau(u,x) m)  =(-1) Y^\tau(u,x) \sigma m=\tilde{Y}(u,x) \sigma m$. 
This proves the claim.

Notice that not every ordinary $V$-module is $\tau$-stable. For instance, untwisted $\W$-modules $\WW(2/5,-)$
is not isomorphic $\WW(2/5,+)$, although the former is obtained as a $\tau$-twist from the later (cf. Theorem \ref{miyamoto}).
Thus when $g^b h^d \neq 1$ we can omit the stability condition (always satisfied).
From now on we are only interested in vacuum twisted characters so we let $u={\bf 1}$. Consider the standard generators $S$ and $T$ of the modular 
group, corresponding to $y \mapsto -1/y$ and $y \mapsto y+1$, respectively. 
For the $S$ matrix $a=0$, $b=1$,  $c=-1$ and $d=0$, and for the $T$-matrix $a=1$, $b=1$, $c=0$ and $d=1$. 

Under the $S$ transformation 
$$T_M(\tau,1,{\bf 1},q)|_{q \cdot S}=\sum_{W} c_W T_W(1,\tau,{\bf 1},q),$$
where the summation is over untwisted modules which are $\tau$-stable. 
Similarly, 
$$T_M(\tau,1 ,{\bf 1},q)|_{q \cdot T}=\sum_{W} d_W T_W(\tau,\tau,{\bf 1},q),$$
where the summation is over $\tau$-twisted modules. 
We also have
$$T_M(\tau,\tau,{\bf 1},q)|_{q \cdot S}=\sum_{W} e_W T_W(\tau,\tau,{\bf 1},q),$$
where the summation is over $\tau$-twisted modules (which are $\tau$-fixed), and 
$$T_M(\tau, \tau, {\bf 1},q)|_{q \cdot T}=\sum_{W} f_W T_W(\tau,1,{\bf 1},q),$$
where the summation is over $\tau$-twisted modules.
Moreover, 
$$T_M(1,\tau,{\bf 1},q)|_{q \cdot S}=\sum_{W} g_W T_W(\tau,1,{\bf 1},q),$$
where the summation is over $\tau$-twisted modules, and 
$$T_M(1 , \tau, {\bf 1},q)|_{q \cdot T}=\sum_{W} h_W T_W(1,\tau,{\bf 1},q),$$
where the summation is over $\tau$-stable $\W$-modules. In the above formulas $c_W,d_W,...,h_W$ are some constants.
We summarize all these relations as 
\begin{corollary} \label{mod-inv} The vector space spanned by 
$$\{T_M(1,\tau ,{\bf 1},q), \ M \  {\rm is} \  {\rm untwisted}, \ {\rm and} \  \tau-{\rm stable}  \}$$
and
$$\{ T_M(\tau,\tau^{\epsilon} ,{\bf 1},q), \ M \  {\rm is} \ \tau-{\rm twisted}, \ {\rm and} \ \epsilon=0,1  \}$$
is modular invariant.
\end{corollary}

Going back to $\W$. There are two $\tau$-stable irreducible untwisted $\W$-modules , namely $\WW(0)$ and $\WW(2/5)$. So the relevant modular invariant 
space has a basis:
\bea
T_{\WW^\tau(1/40)}(\tau,\tau^{\epsilon},{\bf 1},q)&=&\chi_{5,6}^{1,2}(q)+(-1)^\epsilon \chi_{5,6}^{1,4}(q) \nonumber \\
T_{\WW^\tau(1/8)}(\tau,\tau^{\epsilon} ,{\bf 1},q)&=&\chi_{5,6}^{2,2}(q)+(-1)^\epsilon \chi_{5,6}^{2,4}(q),\nonumber 
\eea
where $\epsilon \in \{0,1\}$, and
\bea
&& T_{\WW(0)}(1,\tau,{\bf 1},q)= \chi_{5,6}^{2,1}(q)-\chi_{5,6}^{2,5}(q),  \nonumber \\
&& T_{\WW(2/5)}(1,\tau,{\bf 1},q)=\chi_{5,6}^{1,1}(q)-\chi_{5,6}^{1,5}(q).  \nonumber 
\eea
Thus, the left hand-sides in (\ref{minimal3}) and (\ref{minimal4}) are simply the twisted characters of irreducible (untwisted) $\W$-modules. 
Combined with the expected identities given on the right hand-sides in (\ref{minimal3}) and (\ref{minimal4}),
or simply by using modular invariance arguments,  we easily get another (more natural) basis:

\begin{proposition} For $V=\W$ and $\tau$ as above, the vector space spanned by expressions in Corollary \ref{mod-inv} is $6$-dimensional with a basis 
$$\{  \chi_{2,5}^{1,1}(q^2), \chi_{2,5}^{1,2}(q^2), \chi_{2,5}^{1,1}(q^{1/2}), \chi_{2,5}^{1,2}(q^{1/2}),  \chi_{2,5}^{1,1}(-q^{1/2}), \chi_{2,5}^{1,2}(-q^{1/2})  \}.$$ 
\end{proposition}

\section{Appendix}

In this section we discuss $q$-series identities underlying (\ref{minimal1})-(\ref{minimal4}).
The idea is relatively simple so we only prove (\ref{minimal1}) here, and leave the rest to the reader.
Similar methods can be used to prove other identities

\begin{proposition}
We have
\be \label{min1-proof}
\chi_{5,6}^{1,2}(q)+\chi_{5,6}^{1,4}(q)=\chi_{2,5}^{1,1}(q^{1/2}).
\ee
\end{proposition}

\begin{proof}
The left hand side of (\ref{min1-proof})  is equal to 
$$\frac{q^{11/120} \left(\displaystyle{ \sum_{n \in \ZZ} } q^{30m^2-4m}-q^{30m^2+16m+2}+q^{30m^2-14m+3/2}-q^{30m^2+26m+4+3/2} \right)}{(q)_{\infty}}$$
Thus it is sufficient to prove 
\begin{align}
 \sum_{n \in \ZZ} \left( q^{30m^2-4m}-q^{30m^2+16m+2}+q^{30m^2-14m+3/2}-q^{30m^2+26m+4+3/2} \right) \notag \\
 \label{wanted} =\frac{(q)_\infty}{ \displaystyle{\prod_{n=0}^\infty }(1-q^{(5n+2)/2})(1-q^{(5n+3)/2})}.
 \end{align}
Recall the quintuple product identity 
\begin{align} 
\label{wanted1} \sum_{m \in \ZZ} (-1)^m q^{3m^2+m} z^{3m+1}+\sum_{m \in \ZZ}  (-1)^m q^{3m^2-m} z^{3m} \\
=(1+z)\prod_{n=1}^\infty (1-q^{2n})(1-q^{4n-2}z^2)(1-q^{4n-2} z^{-2})(1+q^{2n}z)(1+q^{2n} z^{-1}) \notag.
\end{align}
We rewrite the left hand-side as 
\bea \label{wanted3}
&& \sum_{m \in \ZZ} q^{12m^2+2m} z^{6m+1}-\sum_{m \in \mathbb{Z}}  q^{12m^2+14m+4} z^{6m+4} \\ 
&&+\sum_{m \in \mathbb{Z}} q^{12m^2-2m} z^{6m} -\sum_{m \in \ZZ} q^{12m^2+10m+2} z^{6m+3}. \nonumber
\eea
Substitute now in (\ref{wanted1})  $q$ for $q^{5/2}$ and $z$ for  $q^{-3/2}$, and multiply the resulting expression with $q^{3/2}$. Then (\ref{wanted1}) turns into the left hand-side of (\ref{wanted}). 
The proof now follows from an easy identity
\bea
&& (1+q^{3/2})\prod_{n=1}^\infty (1-q^{5n})(1-q^{10n-8})(1-q^{10n-2})(1+q^{5n-3/2})(1+q^{5n+3/2}) \nonumber \\
&&=\frac{(q)_\infty}{\displaystyle{\prod_{n=0}^\infty} (1-q^{(5n+2)/2})(1-q^{(5n+3)/2})}. \nonumber
\eea
\end{proof}

\end{document}